\documentclass{amsart}

\newtheorem{theorem}{Theorem}[section]
\newtheorem{lemma}[theorem]{Lemma}

\theoremstyle{definition}

\theoremstyle{remark}
\newtheorem{remark}[theorem]{Remark}

\numberwithin{equation}{section}

\begin{document}

\title{A note on the Harnack inequality for elliptic equations in divergence form}

\author{Dongsheng Li}
\address{}
\curraddr{School of Mathematics and Statistics, Xi'an Jiaotong University, Xi'an 710049, China}
\email{lidsh@mail.xjtu.edu.cn}
\thanks{This research is supported by NSFC 11171266.}

\author{Kai Zhang}
\address{School of Mathematics and Statistics, Xi'an Jiaotong University, Xi'an 710049, China}
\email{zkzkzk@stu.xjtu.edu.cn}

\subjclass[2010]{Primary 35J15, 35B65; Secondary 35D30}

\date{December 14, 2015}


\keywords{Elliptic equations, Harnack inequality}

\begin{abstract}

In 1957, De Giorgi \cite{De1} proved the H\"{o}lder continuity for elliptic equations in divergence form and Moser \cite{Mo1} gave a new proof in 1960. Next year, Moser \cite{Mo2} obtained the Harnack inequality. In this note, we point out that the Harnack inequality was hidden in \cite{De1}.
\end{abstract}

\maketitle

\section{The Harnack inequality}

Consider the following elliptic equation:
\begin{equation}\label{e1.1}
 (a^{ij}u_i)_j=0~~~~\mbox{in}~~Q_6,
\end{equation}
where $(a^{ij})_{n\times n}$ is uniformly elliptic with ellipticity constants $\lambda$ and $\Lambda$. In this note,  $Q_r(x_0)$ denotes the cube with center $x_0$ and side-length $r$, and $Q_r:=Q_r(0)$.

In 1961, Moser \cite{Mo2} obtained the following Harnack inequality:
\begin{theorem}\label{th1.3}
Let $u\geq 0$ be a weak solution of (\ref{e1.1}). Then
\begin{equation}\label{e1.5}
  \sup_{Q_{1}} u \leq C\inf_{Q_{1}} u,
\end{equation}
where $C$ depends only on $n$, $\lambda$ and $\Lambda$.
\end{theorem}
The method used in \cite{Mo2} is that first to estimate the upper and lower bound of $u$ in terms of $||u||_{L^{p_0}}$ and  $||u^{-1}||_{L^{p_0}}$ respectively by an iteration for some
$p_0>0$ and then to join these two estimates together to obtain (\ref{e1.5}) by the John-Nirenberg inequality.

In 1957, De Giorgi \cite{De1} proved the H\"{o}lder continuity for weak solutions of (\ref{e1.1}) and Moser \cite{Mo1} gave a new proof later. The following are two of the main results in \cite{De1} and \cite{Mo1} (see also \cite{De2}):
\begin{theorem}\label{th1.1}
Let $u\geq 0$ be a weak subsolution of (\ref{e1.1}). Then
\begin{equation}\label{e1.2}
  \|u\|_{L^{\infty}(Q_1)}\leq C \|u\|_{L^{2}(Q_3)},
\end{equation}
where $C$ depends only on $n$, $\lambda$ and $\Lambda$.
\end{theorem}

\begin{theorem}\label{th1.2}
Let $u\geq 0$ be a weak supsolution of (\ref{e1.1}). Then for any $c_0>0$, there exists a constant $c$ depending only on $n$, $\lambda$, $\Lambda$ and $c_0$ such that
\begin{equation}\label{e1.3}
  m\{x\in Q_{1}: u(x)>c\}>c_0 \Rightarrow  u>1 \mbox{ in } Q_{3},
\end{equation}
where $m$ denotes the Lebesgue measure.
\end{theorem}
In this note, we will prove Theorem \ref{th1.3} by the above two theorems directly. That is, De Giorgi's proof implies Harnack inequality. This was first noticed by DiBenedetto \cite{P1}. Some other new approaches to Harnack inequality can be founded in \cite{P2} and \cite{P3}, where
U. Gianazza and V. Vespri \cite{P3} requires only a qualitative boundedness of solutions, which is different from here.

\section{Proof of Theorem{\ref{th1.3}}}
In the following, we present the key points for obtaining Theorem \ref{th1.3} by Theorem \ref{th1.1} and \ref{th1.2}. First, we show that Theorem \ref{th1.1} implies the following local maximum principle:
\begin{lemma}\label{le2.1}
Let $u\geq 0$ be a weak subsolution of (\ref{e1.1}) and $p_0>0$. Then
\begin{equation}\label{e1.6}
    \|u\|_{L^{\infty}(Q_{1})}\leq C \|u\|_{L^{p_0}(Q_3)},
\end{equation}
where $C$ depends only on $n$, $\lambda$, $\Lambda$ and $p_0$.
\end{lemma}

\begin{proof}
By the interpolation for $L^p$ functions and \eqref{e1.2}, for any $\varepsilon >0$, we have
\begin{equation*}
    \|u\|_{L^{\infty}(Q_{1})}\leq  \varepsilon \|u\|_{L^{\infty}(Q_{3})}+C(\varepsilon) \|u\|_{L^{p_0}(Q_3)}
\end{equation*}
whose scaling version is
\begin{equation}\label{e2.2}
    r^{\frac{n}{p_0}}\|u\|_{L^{\infty}(Q_{r}(x_0))}\leq  \varepsilon r^{\frac{n}{p_0}}\|u\|_{L^{\infty}(Q_{3r}(x_0))}+C(\varepsilon) \|u\|_{L^{p_0}(Q_{3r}(x_0))}
\end{equation}
for any $Q_{6r}(x_0)\subset Q_6$.

Given $x_0\in Q_3$, denotes by $d_{x_0}$ the distance between $x_0$ and $\partial Q_3$. Then the cube $Q_{6r}(x_0)\subset Q_3$ where $r=d_{x_0}/3\sqrt{n}$, and we have
\begin{equation*}
\begin{aligned}
d_{x_0}^{\frac{n}{p_0}}|u(x_0)|&\leq C(n)r^{\frac{n}{p_0}}\|u\|_{L^{\infty}(Q_{r}(x_0))}\\
&\leq  \varepsilon C(n) r^{\frac{n}{p_0}}\|u\|_{L^{\infty}(Q_{3r}(x_0))}+C(\varepsilon) \|u\|_{L^{p_0}(Q_{3r}(x_0))} (\mbox{By } \eqref{e2.2})\\
&\leq \varepsilon C(n) \sup_{x\in Q_3} d_{x}^{\frac{n}{p_0}}|u(x)|+C(\varepsilon) \|u\|_{L^{p_0}(Q_{3})}.
\end{aligned}
\end{equation*}
Take the supremum over $Q_3$ and  $\varepsilon$ small such that $\varepsilon C(n)<1/2$. Then,
\begin{equation*}
  \sup_{x\in Q_3} d_{x}^{\frac{n}{p_0}}|u(x)|\leq C\|u\|_{L^{p_0}(Q_{3})},
\end{equation*}
which implies \eqref{e1.6}.
\end{proof}

Next, we show that Theorem \ref{th1.2} implies the weak Harnack inequality:
\begin{lemma}\label{le2.2}
  Let $u\geq 0 $ be a weak supsolution of (\ref{e1.1}). Then
\begin{equation}\label{e1.7}
   \|u\|_{L^{p}(Q_{1})}\leq C\inf_{Q_{3}} u,
\end{equation}
where $p>0$ and $C$ depend only on $n$, $\lambda$ and $\Lambda$.
\end{lemma}
\begin{proof}
Without loss of generality, we assume that $\inf_{Q_3} u=1$ and we only need to prove that there exists a constant $c$ depending only on $n$, $\lambda$ and $\Lambda$ such that
\begin{equation}\label{e1.8}
  m\{x\in Q_1: u(x)>c^k\} \leq \frac{1}{2^k},
\end{equation}
for $k=1,2,...$.

We prove (\ref{e1.8}) by induction. Take $c_0=1/2$ in Theorem \ref{th1.2}. Then there exists a constant $c$ depending only on $n$, $\lambda$ and $\Lambda$ such that \eqref{e1.3} holds. Hence, \eqref{e1.8} holds for $k=1$ since we assume that  $\inf_{Q_3} u=1$. Suppose that (\ref{e1.8}) holds for $k\leq k_0-1$. Let
\begin{equation*}
  A:=\{x\in Q_1: u(x)>c^{k_0}\}  \mbox{ and } B:=\{x\in Q_1: u(x)>c^{k_0-1}\}.
\end{equation*}
We need to prove $m(A)\leq  m(B)/2$. By the Calder\'{o}n-Zygmund cube decomposition (see \cite[Lemma 4.2]{C-C}), we only need to prove that for any $Q_{r}(x_0)\subset Q_1$,
\begin{equation*}\label{e1.10}
  m(A\cap Q_r(x_0))>\frac{1}{2} m(Q_r(x_0)) \Rightarrow Q_{3r}(x_0)\cap Q_1\subset B,
\end{equation*}
which is exactly the scaling version of \eqref{e1.3} for $v=u/c^{k_0-1}$.
\end{proof}

Now, Theorem \ref{th1.3} follows clearly by combining (\ref{e1.6}) and (\ref{e1.7}).

\begin{remark}
``$u\geq 0$'' can be removed in Lemma \ref{le2.1} and a corresponding estimate for $u^+$ holds.
As for elliptic equations in non-divergence form, we also have
the local maximum principle (Lemma \ref{le2.1}) and the weak Harnack inequality (Lemma \ref{le2.2}) respectively
(see \cite[Theorem 4.8]{C-C}). In fact, this note is inspired by \cite{C-C}.
\end{remark}

\bibliographystyle{amsplain}

\end{document}